\documentclass[12pt,a4paper,reqno]{amsart}
\usepackage[all]{xy}
\usepackage{anysize}
\usepackage{float}
\usepackage{amsmath}
\usepackage{amssymb}
\usepackage{amsbsy}
\usepackage{amsfonts}
\usepackage{amsthm}
\usepackage{latexsym}
\usepackage{color}

\newtheorem{lem}{Lemma}
\newtheorem{cor}  [lem]{Corollary}
\newtheorem{ex}    [lem]{Example}
\newtheorem{prop}[lem]{Proposition}
\newtheorem{theo}    [lem]{Theorem}

\newtheorem{que} [lem]{Question}

\newtheorem*{theoremA}{Theorem A}

\newtheorem*{theoremB}{Theorem B}

\newcommand{\F}{{\mathbb F}}
\newcommand{\N}{{\mathbb N}}
\newcommand{\Z}{{\mathbb Z}}

\input amssym.def \input amssym.tex

\def\syl#1#2{{\rm Syl}_#1(#2)}
\def\nor{\trianglelefteq\,}

\let\phi=\varphi

\title{A cohomological criterion for $p$-solvability}

\author{Jon Gonz\'alez-S\'anchez}
\address{Departamento de Matem\'aticas,%
Universidad de País Vasco,
Apartado 644, 
48080 Bilbao 
Spain.}
\email{jon.gonzalez@ehu.es}

\author{Joan Tent}
\address{Departament de Matem\`atiques, 
Facultat de Matem\`atiques, 
Universitat de Val\`encia, 
Spain.}
\email{ joan.tent@uv.es}



\thanks{The first author was supported by grants MTM2011-28229-
C02-01 and  MTM2014-53810-C2-2-P, from the Spanish Ministry of Economy and Competitivity,  the Ramon y Cajal Programme of the Spanish Ministry of Science and Innovation, grant RYC-2011-08885, and by the Basque Government, grants IT753-13 and IT974-16. \\ \indent Second author is supported by grants MTM2011-28229-
C02-01 and  MTM2014-53810-C2-2-P, from the Spanish Ministry of Economy and Competitivity, and by Prometeo II/Generalitat Valenciana.}

\begin{document}

\begin{abstract}
Let $G$ be a finite group,  $p$ a prime  and $P$ a Sylow $p$-subgroup  of $G$. In this note we give a cohomological criterion for the $p$-solvability of $G$ depending on  the cohomology in degree $1$
  with coefficients in $\mathbb F_p$  of both the normal subgroups of $G$ and $P$.
As a byproduct we bound the minimal number of quotients of order a power
of $p$ appearing in any normal series of $G$ by the number of generators of $P$.
\end{abstract}

\maketitle

\section{Introduction}

Let $G$ be a finite group and $p$ an arbitrary fixed prime integer.
We denote the $i$-th cohomology and $i$-th homology
groups of $G$ with coefficients in the field $\mathbb F_p$ of $p$ elements by $H^i(G)=H^i(G, \F_p)$ and $H_i(G)=H_i(G, \F_p)$ respectively,
where the action considered is the trivial one.
The first cohomology group is naturally isomorphic to the group of homomorphisms from $G$ into $\mathbb F_p$, that is

$$
H^1(G) \cong {\rm Hom} (G, \mathbb F_p)\, ,
$$  and

$$
H_1(G) \cong H^1(G)^* \cong G/G^p[G,G],
$$ where $A^*$ denotes the dual group of an abelian group $A$,
and $G^p$ is the subgroup of $G$ generated by the $p$-powers of the elements of $G$.

Suppose now that $P$ is a Sylow $p$-subgroup of $G$. Since $P$ has index in $G$
coprime to $p$, the restriction map from $H^1(G)$ into $H^1(P)$
defines an injective group homomorphism, and it is the content of Tate's $p$-nilpotency criterion \cite{Tate} that this injection is a group
isomorphism if and only if $G$ has a normal $p$-complement. In this note, we present
a cohomological criterion for a finite group to be $p$-solvable based on Tate's characterization.
In order to state our main result, we first need to introduce some notation.
If $N\nor G$ and $K\leq G$, then $ K$ acts in a natural way
into $H^1(N)$, and we denote by
 $H^1(N)^K$ the subgroup of fixed points of this action.
 Let  $O^p(G)$ be the smallest  normal subgroup of $G$ such that
$G/O^p(G)$ is a $p$-group, so $H^1(G)\cong H^1(G/O^p(G))$. Similarly, $O^{p'}(G)$
is defined as the minimal normal subgroup of $G$ such that $G/O^{p'}(G)$ has order coprime to $p$.

\begin{theoremA}
Let $G$ be a finite group, $p$ a prime number and $P\in\syl p G$. Write
$M_1=O^{p^\prime} (G)$ and $M_i=O^{p^\prime} (O^p(M_{i-1}))$ for $i\geq 2$. Then the following two conditions are equivalent:
\begin{enumerate}
\item $G$ is $p$-solvable.
\item $H^1(P)\cong \oplus_{i\in \N} H^1(M_i)^P$.\label{cond2}
\end{enumerate}
\end{theoremA}

Observe that it is clear that condition (\ref{cond2}) in Theorem A is equivalent to saying that the groups claimed to be isomorphic have the same dimension as $\mathbb F_p$-vector spaces.

\medskip

Recall that the $p$-length in a $p$-solvable group is defined as the minimal number of quotients of order a power
of $p$ appearing in any normal series of the group.
Note that an immediate consequence of Theorem A is the well-known fact that if $G$ is $p$-solvable,
then the number of generators of a Sylow $p$-subgroup $P$ of $G$ is greater than or equal to the
$p$-length of $G$ (see Huppert's Hauptsatz 6.6 in \cite{Hup}).  We  may extend 
the definition of $p$-length 
to any finite group $G$, by saying that
the \textbf{$p$-length} of $G$ is equal to the minimal number of quotients of order a power
of $p$ appearing in any normal series of $G$. Then we can generalize
Huppert's result as follows:

\begin{theoremB}
Let $G$ be a finite group,  $p$ a prime and
$P\in\syl{p}{G}$. Suppose that $d$ is the number of generators of $P$, and that
$l$ is the $p$-length of $G$.
Then $l\leq d$. 
\end{theoremB}

Another result of Huppert states that for $p$ and odd prime, the number of non-$p$-solvable chief factors of a finite group is bounded by the number of generators of a Sylow $p$-subgroup of the group (see Satz 2.3 of \cite{Hup2}). We will call the number of such chief factors the {\bf non-$p$-solvable length} of the group, noting that this definition is slightly different from the non-$p$-solvable length introduced by E. I. Khukhro  and P. Shumyatsky \cite{Khu-Sum}. Consider
a normal series of a group $G$ whose quotients are either non-$p$-solvable chief factors, $p$-groups or $p^\prime$-groups.
Then we define the {\bf generalized $p$-length} of $G$ as  the smallest possible number of quotients
of order divisible by $p$ appearing in such a series. It is then clear that
the generalized $p$-length of a finite group is bounded by twice
the number $d$ of generators of a Sylow $p$-subgroup of the group.
Clearly, if $G$ is either $p$-solvable or a group with no $p$-solvable composition factors,
then we do not recover, from  the bound $l\leq 2d$ on  the generalized $p$-length $l$
of $G$, the above stated bounds, respectively on the $p$-length and on the non-$p$-solvable length.
Although such a general bound for the generalized $p$-length cannot be obtained,
we have not
been able to set whether our estimation is the best possible of its kind (see Section \ref{last}).

\section{Tate's $p$-nilpotency criterion and a lemma}

The proof of our Theorem A in the Introduction relies on Tate's $p$-nilpotency criterion for finite groups \cite{Tate}.
In this section, we briefly sketch a proof of Tate's well-known result with  a 
small variation on the original arguments, and we also state
some useful consequences of the proof.

\medskip

Fix a prime $p$.
 As before, let $G$ be a finite group and $P$ a Sylow $p$-subgroup of $G$.
Tate's criterion stablishes that the restriction map
$$H^1(G)\longrightarrow H^1(P)$$ is a  group  isomorphism if and only if $G$ has a normal $p$-complement. In order to prove this, let
$N=O^p(G)$.
In particular, note that $N$ has no proper quotients of order a power of $p$, that is $O^p(N)=N$, and thus
$H^1(N)=0$. Write  $M=N\cap P$, so the natural inclusion gives a natural isomorphism from $P/M$ into $G/N$.
Then we have the following commutative diagram, where all the arrows are natural:

\begin{equation}
 \xymatrix{
 1 \ar[r] & M \ar[r]\ar[d] &P\ar[r]\ar[d] & P/M\ar[d]\ar[r] &1 \\
 1 \ar[r] & N \ar[r] &G\ar[r] & G/N\ar[r] &1. \\
 }
\end{equation}
Now the inflation-restriction-transgression exact sequence in cohomology (see \cite[Corollary 7.2.3]{Ev})
leads to the following commutative diagram:
 \begin{equation}
 \label{inf-res-tras}
 \xymatrix{
 0 \ar[r] & H^1(G/N) \ar[r]^{\alpha}\ar[d]^{\iota_1} \ar[d]&H^1(G)\ar[r]^{\beta} \ar[d]^{\iota_2}& H^1(N)^G\ar[r]^{\gamma}\ar[d]^{\iota_3}& H^2(G/N)\ar[r]^{\delta} \ar[d]^{\iota_4}&H^2(G) \ar[d]^{\iota_5}\\
  0 \ar[r] & H^1(P/M) \ar[r]^{\alpha^\prime} &H^1(P)\ar[r]^{\beta^\prime} & H^1(M)^P\ar[r]^{\gamma^\prime}& H^2(P/M)\ar[r]^{\delta^\prime} &H^2(P).
}
\end{equation}
Observe that since $H^1(N)=0$, we have that $\delta$ is a monomorphism. Recall that $\iota_5$ is a monomorphism (because $P$ has index coprime to $p$ in
$G$) and $\iota_4$ is an isomorphism, which implies that $\delta^\prime$ is also a monomorphism.
Therefore $\gamma^\prime$ is the null map and $\beta^\prime$ is surjective.
On the other hand, note that both $\alpha$ and $\iota_1$ are isomorphisms, and by hypothesis $\iota_2$ is also an isomorphism.
Hence $\alpha^\prime$ is an isomorphism and $\beta^\prime$ is the null map.

\medskip

Note that the two observations in the previous paragraph lead to the fact that $H^1(M)^P=0$.
 But if $M$ is non-trivial, since $P$ is a $p$-group, then $H^1(M)^P$ is non-trivial.
This implies that $N$ is a normal $p$-complement of $G$, which is the conclusion of Tate's result.

\medskip

In the above setting, if we drop the hypothesis that $H^1(G)\to H^1(P)$ is an isomorphism,
we conclude the following, which can also be deduced from \cite{Tate} or \cite{Thompson}:

\begin{lem}\label{lemma}
Let $G$ be a finite group and $P\in\syl{p}{G}$. Suppose that $N$ is a normal subgroup of $G$ such that $N=O^p(N)$,
and write $M=N\cap P$. Then the following sequence is exact:
 \begin{equation}
 \xymatrix{
  0 \ar[r] & H^1(P/M) \ar[r]^{\alpha^\prime} &H^1(P)\ar[r]^{\beta^\prime} & H^1(M)^P\ar[r]& 0.
}
\end{equation}
Equivalently, the following sequence is exact:
 \begin{equation}
 \xymatrix{
  0 \ar[r] & H_1(M)_P \ar[r]^{(\beta^\prime)^*} &H_1(P)\ar[r]^{(\alpha^\prime)^*} & H_1(P/M)\ar[r]& 0.
}
\end{equation}
\end{lem}

\begin{proof}
Consider the diagram \eqref{inf-res-tras} for the group $G=PN$. In that situation $\delta$ is injective and $\iota_4$ is an isomorphism.
This implies that $\gamma^\prime$ is is the null map and the lemma follows. \end{proof}

The following well-known fact is a direct consequence of the previous lemma.

\begin{cor}\label{Tate}
Let $G$ be a finite group and $P$ a Sylow $p$-subgroup of $G$.
Suppose that $N$ is a normal subgroup of $G$ such that $P\cap N\leq \Phi(P)$.
Then $N$ is $p$-nilpotent.
\end{cor}

\begin{proof}
Apply the previous lemma to the group $PN$ and its normal subgroup $O^p(N)$.
By hypothesis $H^1(P/(P\cap O^p(N)))\cong H^1(P)$, and therefore $H^1(P\cap O^p(N))^P=0$. Hence $P\cap O^p(N)=1$, so $O^p(N)$
is a normal $p$-complement of $N$.
\end{proof}

\section{A $p$-solvability criterion}\label{main}

In this section we work to prove Theorem A
in the Introduction.

\medskip

We start by proving that ($1$) implies ($2$) in Theorem A,
which follows easily from Lemma \ref{lemma}. Let $G$ be a finite group
and $P\in \syl p G$. Suppose first that
$G$ is $p$-solvable of $p$-length one,
and write $M_1=O^{p'}(G)$, $L_1=O^p(M_1)$.
Observe that since $M_1$ has a normal $p$-complement and $P\subseteq M_1$, it
is clear that restriction defines an isomorphism:

$$
f: H^1(M_1)\longrightarrow H^1(P)\, .
$$

\medskip

\begin{prop}\label{hom}
Let $G$ be a finite $p$-solvable group of $p$-length $l\geq 2$, and $P$ a Sylow $p$-subgroup
of $G$. Let
$L_0=G$, $M_i=O^{p'}(L_{i-1})$ and $L_i=O^p(M_i)$ for $i\geq 1$.
Then the sequences

$$
0\longrightarrow H^1(L_{j-1}P)\stackrel{f_j}{\longrightarrow} H^1(L_{j}P)\stackrel{g_j}{\longrightarrow} H^1(M_j)^P\longrightarrow 0\, ,
$$where the maps $f_j, g_j$ are restriction maps, are exact for $2\leq j \leq l$.
In particular,
we have that $H^1(P)\cong\oplus_{i=1}^l H^1(M_i)^P$.
\end{prop}

\begin{proof} Since any normal subgroup of $G$ of order coprime to $p$ lies in the kernel
of any homomorphism from $G$ into $\mathbb{F}_p$, we can assume that $G$ has no
non-trivial normal $\color{red} p'$-subgroups. We shall prove the result by induction on $l$.

\medskip

Suppose first that $G$ has $p$-length $l=2$, and note that $M_1=L_{1}P$ and $M_2=P\cap L_1$ in this case.
It is clear that $H^1(M_1)\cong H^1(P/M_2)$, so Lemma \ref{lemma} implies that the sequence

$$
0\longrightarrow H^1(M_{1})\stackrel{f_2}{\longrightarrow} H^1(P)\stackrel{g_2}{\longrightarrow} H^1(M_2)^P\longrightarrow 0
$$ is exact, as wanted.

\medskip

We now assume that the result holds for $p$-solvable groups
of $p$-length at most $l-1\geq 2$.
Arguing as in the previous paragraph
with the group $L_{l-1}P$ of $p$-length $2$, we obtain that
the sequence

$$
0\longrightarrow H^1(L_{l-1}P)\stackrel{f_l}{\longrightarrow} H^1(P)\stackrel{g_l}{\longrightarrow} H^1(M_l)^P\longrightarrow 0
$$ is exact. Now, applying the inductive hypothesis on the group $G/L_{l-1}$ we obtain $l-2$
exact sequences, which are easily seen to be the ones needed to complete the proof of the first statement of the proposition.
The statement on the isomorphism of the first cohomology group of $P$ follows directly from this.
\end{proof}

Now it is easy to complete the proof of Theorem A.

\begin{theo}
Let $G$ be a finite group, $p$ a prime number and $P$ a Sylow $p$-subgroup of $G$. Write
$M_1=O^{p^\prime} (G)$ and $M_i=O^{p^\prime} (O^p(M_{i-1}))$ for $i\geq 2$. Then the following two conditions are equivalent:
\begin{enumerate}
\item $G$ is $p$-solvable.\label{cond1}
\item $H^1(P)\cong \oplus_{i\in \N} H^1(M_i)^P$.\label{cond2}
\end{enumerate}
\end{theo}

\begin{proof}
By Proposition \ref{hom}, it only remains to prove that (\ref{cond2}) implies (\ref{cond1}). 
Write $M_0=G$,  so $\{M_i\}_{i\geq 0}$ is a filtration of $G$ that stabilizes at some point, say $t$:

$$M_t=M_{t+1}=M_{t+2}=\ldots$$ 
Observe that $G$ is $p$-solvable if and only if 
$M_t=1$. By the choice of $t$, it is clear that $O^p(M_t)=M_t$. 
Therefore,  by Lemma \ref{lemma} the following sequence is exact:
$$0\longrightarrow H^1(P/M_t\cap P)\longrightarrow H^1(P)\longrightarrow H^1(P\cap M_t)^P\longrightarrow 0.$$
By Proposition \ref{hom}, condition (\ref{cond2}) in the statement and the fact that $P/(M_t\cap P)\cong PM_t/M_t$ 
is a Sylow $p$-subgroup of the $p$-solvable group $G/M_t$, it follows that $H^1(P/M_t\cap P)\cong H^1(P)$. 
Thus $H^1(P\cap M_t)^P=0$, which implies that $P\cap M_t=1$, and this can only occur if $M_t=1$.
\end{proof}

\section{Generalized $p$-length and $p$-perfect groups}\label{last}

In this section we extend some ideas used to prove Theorem A, and give
a proof of Theorem B in the Introduction. At the end of the section, we also propose a conjecture
on
a bound for the generalized $p$-length of a finite group. 

\medskip

Recall that a finite group $G$ is perfect if $H_1(G,\Z)=0$.
Since $H_1(G,\Z_p)$ is the Sylow $p$-subgroup of $H_1(G,\Z)$, it follows that $G$ is perfect if and only if 
the homology group $H_1(G,\Z_p)$ is trivial  for all primes $p$. 
A group satisfying  any  of the following equivalent 
properties 
is called a \textbf{$p$-perfect group}:

\begin{enumerate}
\item $H_1(G,\Z_p)=0$,
\item $H_1(G,\F_p)=0$,
\item $H^1(G,C_{p^\infty})=0$,
\item $H^1(G,\F_p)=0$,
\item $O^p(G)=G$.
\end{enumerate}
We say that a series of normal subgroups $\{N_i\}_{i=0}^r$ of $G$  
is a $p$-perfect filtration if $N_0=G$, 
$N_r=1$, $N_i\leq N_j$ for all $i\geq j$, and for all $i>1$,  the group $N_i$ is $p$-perfect. We define the  \textbf{$p$-perfect length}  of 
$\{N_i\}_{i=0}^r$ as the number of factors $N_{i}/N_{i+1}$ such that $p$ divides $|N_{i+1}:N_i|$. 
We describe the main properties of $p$-perfect filtrations in the following proposition:

\begin{prop}
Let $G$ be a finite $p$-group, $P$ a Sylow $p$-subgroup and $\{N_i\}_{i=0}^r$ a $p$-perfect filtration. Write $M_i=N_i\cap P$  for all i. Then 
\begin{enumerate}
\item  For all $i\geq j$, the restriction map
$$\text{res}^{M_j}_{M_i}:H_1(M_i)_P\longrightarrow H_1(M_j)_P$$
is injective, and it is an isomorphism if and only if $p$ does not divide $|N_j:N_i|$.
\item   For all $i\geq j$, the restriction map
$$\text{res}^{M_j}_{M_i}:H^1(M_j)^P\longrightarrow H^1(M_i)^P$$
is surjective, and it is an isomorphism if and only if $p$ does not divide $|N_j:N_i|$.
\end{enumerate}
\end{prop}

\begin{proof}
Notice that the second statement follows directly from the first one by duality. Let us prove first the injectivity
of the map on the homology groups. Since $N_i$ is $p$-perfect, by Lemma  \ref{lemma} we have the following exact sequence
 \begin{equation}
 \xymatrix{
  0 \ar[r] & H_1(M_i)_{P} \ar[r] &H_1(P)\ar[r]  & H_1(P/M_i)\ar[r]& 0.
}
\end{equation}  In particular, $H_1(M_i)_{P}\to H_1(P)$ is injective.
 Now,  recall that the later arrows factor through $H_1(M_j)_{P}$. That is,
 \begin{equation}
 \xymatrix{H_1(M_i)_{P} \ar[r] &H_1(M_j)_P\ar[r]  & H_1(P).
}
\end{equation}
This shows that the arrow $H_1(M_i)_P\to H_1(M_j)_P$  is injective.

Suppose now that $H_1(M_i)_P\to H_1(M_j)_P$ is an isomorphism. By Lemma  \ref{lemma} applied to the group $N_j$,
 we have the following exact sequence of $P$-modules:
\begin{equation}
 \xymatrix{
  0 \ar[r] & H_1(M_i)_{M_j} \ar[r] &H_1(M_j)\ar[r]  & H_1(M_j/M_i)\ar[r]& 0.
}
\end{equation}
Since taking co-invariants is right exact, the following sequence is exact:
\begin{equation}
 \xymatrix{
 H_1(M_i)_{P} \ar[r] &H_1(M_j)_{P}\ar[r]  & H_1(M_j/M_i)_{P}\ar[r]& 0.
}
\end{equation}
In particular, since $H_1(M_i)_P\to H_1(M_j)_P$ is an isomorphism, we have that $ H_1(M_j/M_i)_{P}=0$. This implies that $M_i=M_j$ and therefore $p$ does not divide $|N_j:N_i|$. The converse is clear.
\end{proof}

The following corollary is straightforward.

\begin{cor}
\label{cor_perfect}
Let $G$ be a finite $p$-group, $P$ a Sylow $p$-subgroup of $G$ and $\{N_i\}_{i=0}^r$ a $p$-perfect filtration. Then the $p$-perfect length of $\{N_i\}_{i=0}^r$ is at most the number of generators of $P$.
\end{cor}

\begin{proof}
Recall that the number of generators of $P$ is the dimension of $H_1(P,\F_p)$. Then the corollary is clear from the previous proposition.
\end{proof}

Next we present a natural way of constructing $p$-perfect filtrations.

\begin{ex}
\label{example_perfect}
Let $G$ be a finite group and $p$ a fixed prime. Write $N_0=G$ and let 
$M_1$ be the $p$-solvable residue of $G$, that is $M_1$ is the smallest
normal subgroup of $G$ such that $G/M_1$ is $p$-solvable. Consider the 
normal series $\{N_j\}_{j=0}^{s_1}$ of $G$ such that $N_0=G$, $N_j=O^p(O^{p^\prime}(N_{j-1}))$ for $j\geq 1$, 
and $N_{s_1}=M$. If $N_{s_1}\neq 1$, next we take $N_{s_1+1}$ a normal subgroup of $G$ such that $N_{s_1}/N_{s_1+1}$ is a chief factor of $G$. 
Notice that by construction $N_{s_1}/N_{s_1+1}$ is non-$p$-solvable, so in particular $p$ divides $|N_{s_1}:N_{s_1+1}|$. Now we take $M_2$ the $p$-solvable residue of $N_{s_1+1}$. Observe that since $M_2$ is characteristic in $N_{s_1+1}$, it is normal in $G$. Then we can proceed to refine the filtration $N_{s_1+1}\supseteq M_2$ as above. 
In this way, and avoiding repetitions if they occur, 
we construct a filtration $\{N_j\}_{j=0}^r$, where $r\geq s_1$, with the following properties:
\begin{itemize}
\item[1.] For all $j$, $N_j/N_{j+1}$ is either a non-$p$-solvable chief factor of $G$ or a $p$-nilpotent group.
\item[2.]  For all $j$, if $N_j/N_{j+1}$ is a 
non-$p$-solvable chief factor, then $N_j$ is $p$-perfect.
\item[3.]  For all $j$, if $N_j/N_{j+1}$ is $p$-nilpotent, 
then $N_{j+1}$ is $p$-perfect.
\end{itemize}
In order to obtain the desired $p$-perfect filtration $\{J_i\}_{i=0}^{t}$ of $G$, 
we can just take the (ordered) subset of $\{N_j\}_{j=0}^r$ formed by the subgroups
$N_j$ such that either $N_{j-1}/N_j$ is $p$-nilpotent, or $N_j/N_{j+1}$ is non-$p$-solvable, 
together with $J_0=G$ and the trivial subgroup $J_t=1$. Of course, it is no loss to assume that we do not have repetitions
in the series. 
\end{ex}

It is clear from the definition given in the Introduction 
that if $G$ is $p$-solvable, then
the $p$-length of $G$ coincides with the generalized $p$-length of $G$.
On the other hand, if $G$ has no $p$-solvable chief factors then the generalized $p$-length of $G$ is just the number of chief factors of
$G$ of order divisible by $p$ that appear in any composition series of $G$. 

\medskip

Now we are ready to prove Theorem B.

\begin{theo}
Let $G$ be a finite group and $P$ a Sylow $p$-subgroup of $G$.
Then the minimal number of generators of $P$ is greater or equal to the $p$-length of $G$.
\end{theo}

\begin{proof}
Consider the filtration $\{J_i\}_{i=0}^t$ of $G$ constructed in Example \ref{example_perfect}. 
The theorem now follows from Corollary \ref{cor_perfect},  because the $p$-perfect $p$-length
of $\{J_i\}_{i=0}^t$ is clearly greater or equal than the $p$-length of $G$.
\end{proof}

It is also easy to deduce from Example \ref{example_perfect}  the following result due to Huppert for odd primes.

\begin{theo}
Let $G$ be a finite group and $P$ a Sylow $p$-subgroup of $G$.
Then the minimal number of generators of $P$ is greater or equal to the non-$p$-solvable $p$-length of $G$.
\end{theo}

After proving this two results, it arises the question whether one could combine both in a more general bound, that is, 
whether the number of generators of the Sylow $p$-subgroups bounds the generalize $p$-length
of a finite group. Unfortunatelly this is not the case for the Schur cover of $S_5$.

\begin{ex}
Let $G$ be a Schur cover of $S_5$. The generalized $2$-length of $G$ is $3$, but the Sylow $2$-subgroup of $G$ is isomorphic to a Schur cover of $D_8$, which can be generated by $2$ elements.
\end{ex}

In any case, it seems that this situation is quite particular, 
and one should expect the following question to be true.

\begin{que}
Is it true that for all, but a finite number of primes, the generalized $p$-length of a finite group is bounded by the number of generators of its Sylow $p$-subgroups?
\end{que}

In order to  give an answer to this question, 
and working with the filtration constructed in example \ref{example_perfect},
one needs to study sequences of $G$-groups:
\begin{equation}
\label{eq-4}
 \xymatrix{
 1\ar[r] &M \ar[r] &\tilde{H}\ar[r]  & H\ar[r]& 1,
}
\end{equation}
where $H$ is a direct product of copies of a finite simple non-abelian group, $M$ is a simple $\F_p[H]$-module and $G$ is a finite group into which
$\tilde{H}$ is embedded as a normal subgroup.
Write $Q$ for a Sylow $p$-subgroup of $G$ and let $P=Q\cap \tilde{H}$. Under these circumstances,
it would be interesting to know when at least one of the following two properties holds:
\begin{enumerate}
\item The sequence
$$
\label{eq-5}
 \xymatrix{
  0 \ar[r] & H^1(P/M) \ar[r] &H^1(P)\ar[r] & H^1(M)^P\ar[r]& 0.
}
$$
is exact.
\item $\dim (H^1(P/M)^Q)\geq 2$.
\end{enumerate}

For instance, if  \eqref{eq-4} is split, an easy argument shows 
that the transgression map $H^1(M)^P\to H^2(P/M)$ is the null map,
and condition $1$ holds. 

\end{document}